\documentclass[10pt,reqno]{amsart}
\usepackage{amsmath,amssymb,amsthm,graphicx,epstopdf,mathrsfs,url}
\usepackage[usenames,dvipsnames]{color}
\usepackage{dsfont}   
\definecolor{darkred}{rgb}{0.7,0.1,0.1}
\definecolor{darkblue}{rgb}{0.1,0.1,0.4}
\definecolor{darkgrey}{rgb}{0.5,0.5,0.5}
\usepackage[colorlinks=true,linkcolor=darkred,citecolor=Blue]{hyperref}

\numberwithin{equation}{section}
\theoremstyle{plain}
\newtheorem{thm}{Theorem}[section]
\newtheorem{lem}[thm]{Lemma}

\newtheorem{prop}[thm]{Proposition}
\newtheorem{cor}[thm]{Corollary}

\theoremstyle{remark}
\newtheorem{remark}[thm]{Remark}

\theoremstyle{plain}

\newcommand{\hyp}[1]{$C^{2}$-hypersurface as in Definition~\ref{definition_hypersurface}}

\newcommand{\sign}{\mathrm{sign}\,}   

\newcommand{\dom}{\mathrm{dom}\,}

\begin{document}
\title[]{A note on the three dimensional Dirac operator with zigzag type boundary conditions}
\author[]{}


\author[M. Holzmann]{Markus Holzmann}
\address{Institut f\"{u}r Angewandte Mathematik\\
Technische Universit\"{a}t Graz\\
 Steyrergasse 30, 8010 Graz, Austria\\
E-mail: {\tt holzmann@math.tugraz.at}}


\maketitle

\begin{center}
{ \em \small Dedicated with great pleasure to my teacher, colleague, and friend \\
Henk de Snoo on the occasion of his 75th birthday.}
\end{center}

\begin{abstract}
  In this note the three dimensional Dirac operator $A_m$ with boundary conditions, which are the analogue of the two dimensional zigzag boundary conditions, is investigated. It is shown that $A_m$ is self-adjoint in $L^2(\Omega;\mathbb{C}^4)$ for any open set $\Omega \subset \mathbb{R}^3$ and its spectrum is  described explicitly in terms of the spectrum of the Dirichlet Laplacian in $\Omega$. In particular, whenever the spectrum of the Dirichlet Laplacian is purely discrete, then also the spectrum of $A_m$ consists of discrete eigenvalues that accumulate at $\pm \infty$ and one additional eigenvalue of infinite multiplicity.
\end{abstract}

\section{Introduction}

In the recent years Dirac operators with boundary conditions, which make them self-adjoint, gained a lot of attention. From the physical point of view, they appear in various applications such as in the description of relativistic particles that are confined in a box $\Omega \subset \mathbb{R}^3$; in this context the MIT bag model is a particularly interesting example, cf. \cite{ALTR17}. Moreover, in space dimension two the spectral properties of self-adjoint massless Dirac operators play an important role in the mathematical description of graphene, see, e.g., \cite{BFSB17_2} and the references therein. On the other hand, from the mathematical point of view, self-adjoint Dirac operators with boundary conditions are viewed as the relativistic counterpart of Laplacians with Robin type and other boundary conditions.

To set the stage, let $\Omega \subset \mathbb{R}^3$ be an open set, let
\begin{equation} \label{def_Pauli_matrices}
   \sigma_1 := \begin{pmatrix} 0 & 1 \\ 1 & 0 \end{pmatrix}, \qquad
   \sigma_2 := \begin{pmatrix} 0 & -i \\ i & 0 \end{pmatrix}, \qquad
   \sigma_3 := \begin{pmatrix} 1 & 0 \\ 0 & -1 \end{pmatrix},
\end{equation}
be the Pauli spin matrices, and let
\begin{equation} \label{def_Dirac_matrices}
   \alpha_j := \begin{pmatrix} 0 & \sigma_j \\ \sigma_j & 0 \end{pmatrix}
   \quad \text{and} \quad \beta := \begin{pmatrix} I_2 & 0 \\ 0 & -I_2 
\end{pmatrix}
\end{equation}
be the $\mathbb{C}^{4 \times 4}$-valued Dirac matrices, where $I_d$ denotes the identity matrix in $\mathbb{C}^{d \times d}$. For $m \in \mathbb{R}$ we introduce the differential operator $\tau_m$ acting on distributions by
\begin{equation} \label{Dirac_operator_intro}
  \tau_m := -i \sum_{j=1}^3 \alpha_j \partial_j + m \beta =: -i \alpha \cdot \nabla + m \beta.
\end{equation}
The main goal in this short note is to study the self-adjointness and the spectral properties of the Dirac operator $A_m$ in $L^2(\Omega; \mathbb{C}^4) := L^2(\Omega) \otimes \mathbb{C}^4$ which acts as $\tau_m$ on functions $f = (f_1, f_2, f_3, f_4) \in L^2(\Omega; \mathbb{C}^4)$ which satisfy $\tau_m f \in L^2(\Omega; \mathbb{C}^4)$ and the boundary conditions 
\begin{equation} \label{boundary_conditions_intro}
  f_3 |_{\partial \Omega} = f_4|_{\partial \Omega} = 0;
\end{equation}
for irregular or unbounded domains~\eqref{boundary_conditions_intro} is understood as $f_3,f_4 \in H^1_0(\Omega)$. Note that no boundary conditions are imposed for the components $f_1$ and $f_2$. If $m>0$, then the solution of the evolution equation with Hamiltonian $A_m$ describes the propagation of a quantum particle with mass $m$ and spin $\frac{1}{2}$ in $\Omega$ taking these boundary conditions and relativistic effects into account.

The motivation to study the operator $A_m$ is twofold. Firstly, in the recent paper \cite{BHM19} Dirac operators in $L^2(\Omega; \mathbb{C}^4)$ acting as $\tau_m$ on functions satisfying the boundary conditions
\begin{equation} \label{boundary_conditions_BHM}
   \vartheta \big(I_4 + i \beta (\alpha \cdot \nu)\big) f|_{\partial \Omega} 
      = \big(I_4 + i \beta (\alpha \cdot \nu)\big) \beta f|_{\partial \Omega}
\end{equation}
were studied in the case that $m> 0$ and that $\Omega$ is a $C^2$-domain with compact boundary and unit normal vector field $\nu$; in~\eqref{boundary_conditions_BHM} the convention $\alpha \cdot x = \alpha_1 x_1 + \alpha_2 x_2 + \alpha_3 x_3$ is used for $x = (x_1,x_2,x_3) \in \mathbb{R}^3$. The authors were able to prove the self-adjointness and to derive the basic spectral properties of these operators, whenever the parameter $\vartheta$ appearing in~\eqref{boundary_conditions_BHM} is a real-valued H\"older continuous function of order $a > \frac{1}{2}$ satisfying $\vartheta(x) \neq \pm 1$ for all $x \in \partial \Omega$, and it is shown that the domain of definition of these self-adjoint operators is contained in the Sobolev space $H^1(\Omega; \mathbb{C}^4)$. For bounded domains $\Omega$ this implies, in particular, that the spectrum is purely discrete. 
The case when $\vartheta(x) = \pm 1$ for some $x \in \partial \Omega$ remained open and it is conjectured that different spectral properties should appear. We note that $\vartheta \equiv 1$ corresponds to the boundary conditions~\eqref{boundary_conditions_intro}.
Let us mention that the self-adjointness and spectral properties of Dirac operators with boundary conditions of the form~\eqref{boundary_conditions_BHM} for special realizations $\vartheta \neq 1$ were studied in 3D in \cite{ALTMR19, ALTR17, OV18} and in 2D in \cite{BFSB17_1, BFSB17_2, LTO18, LO18}.

The second main motivation for this study is the paper \cite{S95}, where the two dimensional counterpart of $A_m$ was investigated in the massless case ($m=0$). The two dimensional Dirac operator is a differential operator in $L^2(\Omega; \mathbb{C}^2)$ and for a spinor $f = (f_1,f_2)$ the zigzag boundary conditions are $f_2=0$ on $\partial \Omega$, while there are no boundary conditions for $f_1$, cf. \cite{BFSB17_1, S95}. The two dimensional zigzag boundary conditions have a physical relevance, as they appear in the description of graphene quantum dots, when a lattice in this quantum dot is terminated and the direction of the boundary is perpendicular to the bonds \cite{GPKH14}.
It was shown in \cite{S95} that the two dimensional Dirac operator with these zigzag boundary conditions is self-adjoint on a domain which is in general not contained in $H^1(\Omega)$ and that for any bounded domain $\Omega$ zero is an eigenvalue with infinite multiplicity. In particular, the spectrum of the operator is not purely discrete. 

The goal in the present note is to prove similar and even more explicit results as those in \cite{S95} also in the three dimensional setting, which complement then the results from \cite{BHM19} in the critical case $\vartheta \equiv 1$ at least for constant boundary parameters. In Lemma~\ref{lemma_theta_minus_one} we will see that the operator corresponding to $\vartheta \equiv -1$ in~\eqref{boundary_conditions_BHM} is unitarily equivalent to $-A_m$ and hence this case is also contained in the analysis in this note. In the formulation of the following main result of the present paper we denote by $-\Delta_D$ the self-adjoint realization of the Dirichlet Laplacian in $L^2(\Omega)$.

\begin{thm} \label{theorem_intro}
  The operator $A_m$ is self-adjoint in $L^2(\Omega; \mathbb{C}^4)$ and its spectrum is 
  \begin{equation*}
    \sigma(A_m) = \{ m \} \cup \left\{ \pm \sqrt{\lambda + m^2}: \lambda \in \sigma(-\Delta_D) \right\}.
  \end{equation*}
  The value $m$ always belongs  to the essential spectrum of $A_m$, while for $m\neq 0$ the number $-m$ is not an eigenvalue of $A_m$. Moreover, for $\lambda > 0$ the numbers $\pm \sqrt{\lambda + m^2}$ are both eigenvalues of $A_m$ with multiplicity $2k$ if and only if $\lambda$ is an eigenvalue of $-\Delta_D$ with multiplicity $k$. 
\end{thm}

The proof of Theorem~\ref{theorem_intro} is done in several steps in Section~\ref{section_self_adjoint}. It is used that $A_m$ has a supersymmetric structure. Some properties of supersymmetric operators which are needed in this paper are collected in Appendix~\ref{appendix}.

Theorem~\ref{theorem_intro} gives a full description of the spectrum of $A_m$ in terms of the spectrum of the Dirichlet Laplacian in $\Omega$, which is well-studied in many cases. For bounded domains $\Omega$ it follows from the Rellich embedding theorem that the spectrum of $-\Delta_D$ is purely discrete and therefore, the spectrum of $A_m$ consists of an infinite sequence of discrete eigenvalues accumulating at $\pm \infty$ and the eigenvalue $m$, which has infinite multiplicity. In particular, the essential spectrum of $A_m$ is not empty, which is in contrast to the case of non-critical boundary values in \cite{BHM19}. Moreover, if $\Omega$ is a bounded Lipschitz domain, then the non-emptiness of the essential spectrum implies that the domain of $A_m$ is not contained in the Sobolev space $H^s(\Omega; \mathbb{C}^4)$ for any $s>0$. The above results are discussed in a more detailed way at the end of Section~\ref{section_self_adjoint}.

If $\Omega$ is unbounded, then there are different ways how the spectrum of the Dirichlet Laplacian and hence also the spectrum of $A_m$ may look like. On the one hand it is known that for some special horn shaped domains $\Omega$, which have infinite measure, the spectrum of $-\Delta_D$ is purely discrete, cf. \cite{R48, S83}. Therefore, by Theorem~\ref{theorem_intro} also in this case the spectrum of $A_m$  consists only of eigenvalues and it follows from the spectral theorem that $m$ is an eigenvalue with infinite multiplicity. 
On the other hand, for many unbounded domains it is known that $\sigma(-\Delta_D) = [0, \infty)$ and thus, $\sigma(A_m)=(-\infty,-|m|] \cup [|m|, \infty)$ for such $\Omega$. The simplest example for this case is when $\Omega$ is the complement of a bounded domain.

Let us finally collect some basic notations that are frequently used in this note. If not stated differently $\Omega$ is an arbitrary  open subset of $\mathbb{R}^3$. For $n \in \mathbb{N}$ we write $L^2(\Omega; \mathbb{C}^n) := L^2(\Omega) \otimes \mathbb{C}^n$. The inner product and the norm in $L^2(\Omega; \mathbb{C}^n)$ are denoted by $(\cdot, \cdot)$ and $\| \cdot \|$, respectively. We use for $k \in \mathbb{N}$ the symbol $H^k(\Omega)$ for  the $L^2$-based Sobolev spaces of $k$ times weakly differentiable functions and $H^1_0(\Omega)$ for the closure of the test functions $C^\infty_0(\Omega)$ in $H^1(\Omega)$. For a linear operator $A$ its domain is $\dom A$ and its Hilbert space adjoint is denoted by $A^*$. If $A$ is a closed operator, then $\sigma(A)$ and $\sigma_\text{p}(A)$ are the spectrum and the point spectrum of $A$, respectively, and if $A$ is self-adjoint, then its essential spectrum is $\sigma_\text{ess}(A)$.

\section{Some auxiliary operators}

In this section we introduce and discuss two auxiliary operators $\mathcal{T}_\text{min}$ and $\mathcal{T}_\text{max}$ in $L^2(\Omega; \mathbb{C}^2)$ which will be useful to study the Dirac operator $A_m$ with zigzag type boundary conditions. Let $\Omega \subset \mathbb{R}^3$ be an arbitrary open set and let $\sigma = (\sigma_1, \sigma_2, \sigma_3)$ be the Pauli spin matrices defined by~\eqref{def_Pauli_matrices}. In the following we will often use the notation $\sigma \cdot \nabla = \sigma_1 \partial_1 + \sigma_2 \partial_2 + \sigma_3 \partial_3$. We define the set $\mathcal{D}_\text{max} \subset L^2(\Omega; \mathbb{C}^2)$ by
\begin{equation*} 
  \mathcal{D}_\text{max} := \big\{ f \in L^2(\Omega; \mathbb{C}^2): (\sigma \cdot \nabla) f \in L^2(\Omega; \mathbb{C}^2) \big\},
\end{equation*}
where the derivatives are understood in the distributional sense,
and the operators $\mathcal{T}_\text{max}$ and $\mathcal{T}_\text{min}$ acting in $L^2(\Omega; \mathbb{C}^2)$ by
\begin{equation} \label{def_T_max}
  \mathcal{T}_\text{max} f := -i (\sigma \cdot \nabla) f, \quad \dom \mathcal{T}_\text{max} = \mathcal{D}_\text{max},
\end{equation}
and $\mathcal{T}_\text{min} := \mathcal{T}_\text{max} \upharpoonright H^1_0(\Omega; \mathbb{C}^2)$, which has the more explicit representation
\begin{equation} \label{def_T_min}
  \mathcal{T}_\text{min} f := -i (\sigma \cdot \nabla) f, \quad \dom \mathcal{T}_\text{min} = H^1_0(\Omega; \mathbb{C}^2). 
\end{equation}
In the following lemma we summarize the basic properties of $\mathcal{T}_\text{min}$ and $\mathcal{T}_\text{max}$. 

\begin{lem} \label{lemma_T_min_max}
  The operators $\mathcal{T}_\textup{min}$ and $\mathcal{T}_\textup{max}$ are both closed and adjoint to each other, i.e. $\mathcal{T}_\textup{min}^* = \mathcal{T}_\textup{max}$. Moreover, the inclusion $\mathcal{D}_\textup{max} \subset H^1_\textup{loc}(\Omega; \mathbb{C}^2)$ holds.
\end{lem}
\begin{proof}
  The facts that $\mathcal{T}_\text{min}$ and $\mathcal{T}_\text{max}$ are closed and adjoint to each other are simple to obtain by replacing $\alpha \cdot \nabla$ by $\sigma \cdot \nabla$ in the proof of \cite[Proposition~3.1]{BH19} or \cite[Proposition~2.10]{OV18}, see also \cite[Lemma~2.1]{BFSB17_1} for similar arguments.  Furthermore, $\mathcal{D}_\textup{max} \subset H^1_\textup{loc}(\Omega; \mathbb{C}^2)$ can be proved similarly as  \cite[Proposition~1]{S95}.
\end{proof}

Eventually, we show that $0$ always belongs to the essential spectrum of $\mathcal{T}_\text{min} \mathcal{T}_\text{max}$. This result will be of importance to prove that $m$ is in the essential spectrum of $A_m$. 

\begin{prop} \label{proposition_spectrum_T_max}
  There exists a sequence $(f_n) \subset \dom \mathcal{T}_\textup{max}$ with $\| f_n \| = 1$ converging weakly to zero such that $\mathcal{T}_\textup{max} f_n \rightarrow 0$, as $n \rightarrow \infty$. In particular, one has $0 \in \sigma_\textup{ess}(\mathcal{T}_\textup{min} \mathcal{T}_\textup{max})$.
\end{prop}
\begin{proof}
  We distinguish two cases for $\Omega$. First, assume that $g_n(x) := (x_1+ix_2)^n$, $x = (x_1,x_2,x_3) \in \Omega$, belongs to $L^2(\Omega; \mathbb{C})$ for all $n \in \mathbb{N}$. Then, we follow ideas from \cite[Proposition~2]{S95} and see that the functions
  \begin{equation*}
    f_n := \frac{1}{\| g_n \|} \begin{pmatrix} g_n \\ 0 \end{pmatrix} \in L^2(\Omega; \mathbb{C}^2)
  \end{equation*}
  fulfil $(\sigma \cdot \nabla) f_n = 0$, i.e. $f_n \in \ker \mathcal{T}_\text{max}$. Hence, zero is an eigenvalue of infinite multiplicity, which implies immediately the claim.

  In the other case, when $g_n(x) = (x_1+ix_2)^n$, $x=(x_1,x_2,x_3)\in \Omega$, does not belong to $L^2(\Omega; \mathbb{C})$ for some $n \in \mathbb{N}$, we follow ideas from the appendix of \cite{DS92}, where it is shown that zero always belongs to the essential spectrum of the Neumann Laplacian in $L^2(G; \mathbb{C})$, when the domain $G$ has infinite measure, to construct the sequence $(f_n)$. Let $k \in \mathbb{N}$ be the smallest number such that $g_k \notin L^2(\Omega; \mathbb{C})$. 
  Define the Borel measure $\mu$ acting on Borel sets $\mathcal{B} \subset \mathbb{R}^3$ as
  \begin{equation*}
    \mu(\mathcal{B}) := \int_{\mathcal{B}} |g_k(x)|^2 \text{d} x = \int_{\mathcal{B}} (x_1^2+x_2^2)^k \text{d} x
  \end{equation*}
  and the sets $\Omega_n := \{ x \in \Omega: |x| \leq n \}$. Then by assumption $\mu(\Omega)=\infty$ and $\mu(\Omega_n) \leq c n^{2k+3}$. Next, define for $n \in \mathbb{N}$ the functions
  \begin{equation*}
    h_n(x) := \begin{cases} g_k(x), & x \in \Omega_{n-1}, \\ (n-|x|) g_k(x), & x \in \Omega_n \setminus \Omega_{n-1}, \\ 0, & x \in \Omega \setminus \Omega_n, \end{cases}
  \end{equation*}
  and
  \begin{equation*}
    f_n := \frac{1}{\| h_n\|} \begin{pmatrix} h_n \\ 0 \end{pmatrix} \in H^1(\Omega; \mathbb{C}^2).
  \end{equation*}
  Since $(\sigma \cdot \nabla) \binom{g_k}{0} = 0$ one has
  \begin{equation*}
    \| \mathcal{T}_\text{max} f_n \|^2 = \frac{\big\| \mathcal{T}_\text{max} \binom{h_n}{0} \big\|^2}{\|h_n\|^2} \leq \frac{\mu(\Omega_n \setminus \Omega_{n-1})}{\mu(\Omega_{n-1})} =: \alpha_n.
  \end{equation*}
  We claim that $\liminf_{n \rightarrow \infty} \alpha_n = 0$, which implies that there exists a subsequence of $(f_n)$, that is still denoted by $(f_n)$, converging weakly to zero (as $\| h_n \| \rightarrow \mu(\Omega)=\infty$ for $n \rightarrow \infty$) such that $\| \mathcal{T}_\text{max} f_n\| \rightarrow  0$, as $n \rightarrow \infty$, and hence the claim of this proposition also in the second case.
  
  If $\liminf_{n \rightarrow \infty} \alpha_n \neq 0$, then there exists $\alpha > 0$ such that $\alpha_n \geq \alpha$ for almost all $n \in \mathbb{N}$. In particular, this implies
  \begin{equation*}
    \mu(\Omega_n) = \frac{\mu(\Omega_n \setminus \Omega_{n-1}) + \mu(\Omega_{n-1})}{\mu(\Omega_{n-1})} \mu(\Omega_{n-1}) \geq (1+\alpha) \mu(\Omega_{n-1})
  \end{equation*}
  and, by repeating this argument, $\mu(\Omega_n) \geq \widetilde{c} (1+\alpha)^{n-1}$ for a constant $\widetilde{c}>0$. However, this violates the condition $\mu(\Omega_n) \leq c n^{2k+3}$. Thus, $\liminf_{n \rightarrow \infty} \alpha_n = 0$.
  
  Finally, in all cases we have shown that there exists a sequence $(f_n) \subset \dom \mathcal{T}_\textup{max}$ with $\| f_n \| = 1$ converging weakly to zero such that $\mathfrak{b}[f_n] \rightarrow 0$ for $n \rightarrow \infty$, where $\mathfrak{b}$ is the closed quadratic form
  \begin{equation*}
    \mathfrak{b}[g] := \| \mathcal{T}_\text{max} g \|^2, \qquad \dom \mathfrak{b} = \dom \mathcal{T}_\text{max}.
  \end{equation*}
  Since the form $\mathfrak{b}$ is associated to the non-negative self-adjoint operator $\mathcal{T}_\text{min} \mathcal{T}_\text{max}$ via the first representation theorem, the properties of $(f_n)$ and the min-max principle imply that $0 \in \sigma_\text{ess}(\mathcal{T}_\text{min} \mathcal{T}_\text{max})$. This finishes the proof of this proposition.
%
%
\end{proof}

\section{Definition of $A_m$ and its spectral properties} \label{section_self_adjoint}

This section is devoted to the study of the operator $A_m$ and the proof of the main result of this note, Theorem~\ref{theorem_intro}. First, we introduce $A_m$ rigorously and show its self-adjointness, then we investigate its spectral properties.

Let $\Omega \subset \mathbb{R}^3$ be an arbitrary open set and let $\mathcal{T}_\text{max}$ and $\mathcal{T}_\text{min}$ be the operators defined in~\eqref{def_T_max} and~\eqref{def_T_min}, respectively.
We define for $m \in \mathbb{R}$ the Dirac operator $A_m$ with zigzag type boundary conditions, which acts in $L^2(\Omega; \mathbb{C}^4)$, by
\begin{equation} \label{def_A_m}
  A_m = \begin{pmatrix} m I_2 & \mathcal{T}_\text{min} \\ \mathcal{T}_\text{max} & -m I_2 \end{pmatrix}.
\end{equation}
The operator in~\eqref{def_A_m} is the rigorous mathematical object associated to the expression in~\eqref{Dirac_operator_intro} with the boundary conditions~\eqref{boundary_conditions_intro}. We note that Lemma~\ref{lemma_T_min_max} implies that $\dom A_m \subset H^1_\text{loc}(\Omega; \mathbb{C}^4)$. Moreover, since $\mathcal{T}_\text{min}$ is closed and $\mathcal{T}_\text{min}^*=\mathcal{T}_\text{max}$, we see that $A_m$ is supersymmetric in the sense of~\eqref{def_D}.

\begin{remark}
If $\Omega$ is a $C^2$-domain with compact boundary, then there exists a Dirichlet trace operator on $\mathcal{D}_\text{max}$ and one can show with the help of \cite[Propositions~2.1 and~2.16]{OV18} that the expressions in~\eqref{Dirac_operator_intro}--\eqref{boundary_conditions_intro} and~\eqref{def_A_m} indeed coincide. 
\end{remark}

Before we start analyzing $A_m$ we remark that this operator is unitarily equivalent with the operator $-B_{m}$, where $B_m$ is defined by
\begin{equation*} 
  B_m = \begin{pmatrix} m I_2 & \mathcal{T}_\text{max} \\ \mathcal{T}_\text{min} & -m I_2 \end{pmatrix}.
\end{equation*}
Note that $B_m$ is the Dirac operator acting on spinors $f = (f_1,f_2,f_3,f_4) \in L^2(\Omega; \mathbb{C}^4)$ satisfying the boundary conditions $f_1|_{\partial \Omega} = f_2|_{\partial \Omega} = 0$. In particular, the following Lemma~\ref{lemma_theta_minus_one} shows that all results which are proved in this paper for $A_m$ can simply be  translated to corresponding results for $B_m$. In order to formulate the lemma we recall the definition of the Dirac matrix $\beta$ from~\eqref{def_Dirac_matrices}, define the matrix
\begin{equation*}
  \gamma_5 = \begin{pmatrix} 0 & I_2 \\ I_2 & 0 \end{pmatrix},
\end{equation*}
and note that $\beta \gamma_5$ is a unitary matrix.

\begin{lem} \label{lemma_theta_minus_one}
  Set $\mathcal{M} := \beta \gamma_5$. Then $B_m = -\mathcal{M} A_m \mathcal{M}$ holds. In particular, $B_m$ is unitarily equivalent to $-A_m$.
\end{lem}

Now we start analyzing $A_m$. First, we discuss its self-adjointness. This property follows from the supersymmetry of $A_m$ and the abstract result in Proposition~\ref{proposition_susy_self_adjoint}.

\begin{thm} \label{theorem_self_adjoint}
  The operator $A_m$ is self-adjoint in $L^2(\Omega; \mathbb{C}^4)$.
\end{thm}
\begin{proof}
  Since $\mathcal{T}_\text{max} = \mathcal{T}_\text{min}^*$ and $\mathcal{T}_\text{min}$ is closed  by Lemma~\ref{lemma_T_min_max}, the operator $A_m$ has a supersymmetric structure as in~\eqref{def_D}. Hence, the self-adjointness of $A_m$ follows immediately from Proposition~\ref{proposition_susy_self_adjoint}.
\end{proof}

In the following theorem we state the spectral properties of $A_m$. We will see that they are closely related to the spectral properties of the Dirichlet Laplacian $-\Delta_D$, which is the self-adjoint operator in $L^2(\Omega; \mathbb{C})$ that is associated to the closed and non-negative sesquilinear form
\begin{equation} \label{def_Dirichlet_form}
  \mathfrak{a}_D[f,g] := \int_\Omega \nabla f \cdot \overline{\nabla g} \, \text{d} x, \quad \dom \mathfrak{a}_D = H^1_0(\Omega; \mathbb{C}).
\end{equation}
In order to prove this result, we employ again the supersymmetric structure of $A_m$ and the abstract results formulated in Appendix~\ref{appendix}.

\begin{thm} \label{theorem_spectrum_general}
  For any $m \in \mathbb{R}$ the following is true:
  \begin{itemize}
    \item[$\textup{(i)}$] All eigenvalues of $A_m$ have even multiplicity.
    \item[$\textup{(ii)}$] $m \in \sigma_\textup{ess}(A_m)$.
    \item[$\textup{(iii)}$] If $m \neq 0$, then $-m \notin \sigma_\textup{p}(A_m)$.
    \item[$\textup{(iv)}$] Let $-\Delta_D$ be the Dirichlet Laplacian on $\Omega$. Then 
    \begin{equation*}
      \sigma(A_m) = \{ m \} \cup \left\{ \pm \sqrt{\lambda + m^2}: \lambda \in \sigma(-\Delta_D) \right\}
    \end{equation*}
    and for $\lambda > 0$ one has $\pm \sqrt{\lambda + m^2} \in \sigma_\textup{p}(A_m)$ with multiplicity $2k$ if and only if $\lambda \in \sigma_\textup{p}(-\Delta_D)$ with multiplicity $k$. In particular, $\sigma(A_m) \cap (-|m|, |m|) = \emptyset$.
  \end{itemize}
\end{thm}

We note that Theorem~\ref{theorem_spectrum_general} applied for $m=0$ shows that the spectrum of $A_0$ is symmetric w.r.t. $\lambda=0$. This observation would also follow from the stronger fact that $A_0 = -\beta A_0 \beta$, i.e. $A_0$ is unitarily equivalent to $-A_0$.

\begin{proof}[Proof of Theorem~\ref{theorem_spectrum_general}]
%
  (i) Consider the nonlinear time reversal operator
  \begin{equation*}
    T f := -i \gamma_5 \alpha_2 \overline{f}, \qquad f \in L^2(\mathbb{R}^3; \mathbb{C}^4), \quad 
    \gamma_5 := \begin{pmatrix} 0 & I_2 \\ I_2 & 0 \end{pmatrix}.
  \end{equation*}
  One has $f \in \dom A_m$ if and only if $T f \in \dom A_m$ and $T^2 f = -f$. Let $\lambda \in \sigma_\text{p}(A_m)$ and let $f_\lambda$ be a corresponding eigenfunction. Then, one can show in the same way as in \cite[Proposition~4.2~(ii)]{BEHL19_1} that also $T f_\lambda$ is a linearly independent eigenfunction of $A_m$ for the eigenvalue $\lambda$. This shows the claim of statement~(i).
  
  In order to prove statements~(ii)--(iv), we use that $A_m$ has a supersymmetric structure and employ Proposition~\ref{proposition_susy_spectrum}. Indeed, we have $\mathcal{T}_\text{max} = \mathcal{T}_\text{min}^*$ and $\mathcal{T}_\text{min}$ is closed by Lemma~\ref{lemma_T_min_max} and hence the operator $A_m$ has a supersymmetric structure as in~\eqref{def_D}. Thus, we are allowed to use Proposition~\ref{proposition_susy_spectrum} for $S = \mathcal{T}_\text{min}$ to characterize the spectrum of $A_m$. For this purpose, we show first that $\mathcal{T}_\text{max} \mathcal{T}_\text{min} = -\Delta_D I_2$. To see this, we note that $\mathcal{T}_\text{max} \mathcal{T}_\text{min}$ is the unique self-adjoint operator corresponding to the closed quadratic form
  \begin{equation*}
    \mathfrak{a}[f] := \|\mathcal{T}_\text{min} f\|^2, \quad f \in \dom \mathfrak{a} = \dom \mathcal{T}_\text{min}= H^1_0(\Omega; \mathbb{C}^2).
  \end{equation*}
  Since $\sigma_j \sigma_k + \sigma_k \sigma_j = 2 \delta_{j k} I_2$ holds by the definition of the Pauli matrices in~\eqref{def_Pauli_matrices}, we have for $f \in C_0^\infty(\Omega; \mathbb{C}^2)$
  \begin{equation*}
    \mathfrak{a}[f] = \big(-i (\sigma \cdot \nabla )f, -i (\sigma \cdot \nabla )f \big)
        = \big(f, -(\sigma \cdot \nabla )^2f \big) = \big(f, -\Delta f \big) = \| \nabla f \|^2,
  \end{equation*}
  which extends by density to all $f \in \dom \mathfrak{a} = H^1_0(\Omega; \mathbb{C}^2)$. Therefore, $\mathfrak{a}$ is the quadratic form associated to $-\Delta_D I_2$ and hence, by the first representation theorem we conclude $\mathcal{T}_\text{max} \mathcal{T}_\text{min} = - \Delta_D I_2$.
  
  Now we are prepared to prove items (ii)--(iv). Since $\mathcal{T}_\text{max} = \mathcal{T}_\text{min}^*$ by Lemma~\ref{lemma_T_min_max} we get $\sigma(\mathcal{T}_\text{min} \mathcal{T}_\text{max}) \setminus \{ 0 \} = \sigma(\mathcal{T}_\text{max} \mathcal{T}_\text{min}) \setminus \{ 0 \}$, cf. \cite[Corollary~5.6]{T92}, and thus Proposition~\ref{proposition_susy_spectrum} implies that
  \begin{equation} \label{spectrum_A_m_1}
    \left\{ \pm \sqrt{\lambda + m^2}: \lambda \in \sigma(-\Delta_D) \right\} \subset \sigma(A_m) \subset \{ \pm m \} \cup \left\{ \pm \sqrt{\lambda + m^2}: \lambda \in \sigma(-\Delta_D) \right\}.
  \end{equation}
  As $0 \notin \sigma_\text{p}(-\Delta_D)$, we conclude from Proposition~\ref{proposition_susy_spectrum}~(ii) that $-m \notin \sigma_\text{p}(A_m)$ for $m \neq 0$ and thus assertion~(iii) of this theorem. Moreover, the results from Propositions~\ref{proposition_spectrum_T_max} and~\ref{proposition_susy_spectrum} imply that $m \in \sigma_\text{ess}(A_m)$, i.e. point~(ii). For the final observation we use again \cite[Corollary~5.6]{T92}, which implies that $\sigma_\text{p}(\mathcal{T}_\text{min} \mathcal{T}_\text{max}) \setminus \{ 0 \} = \sigma_\text{p}(\mathcal{T}_\text{max} \mathcal{T}_\text{min}) \setminus \{ 0 \}$ and that the multiplicities of the eigenvalues coincide. Thus, Proposition~\ref{proposition_susy_spectrum} implies for $\lambda>0$ that $\pm \sqrt{\lambda+m^2} \in \sigma_\text{p}(A_m)$ with multiplicity $2k$ (recall that all eigenvalues of $A_m$ have even multiplicities by (i)) if and only if $\lambda \in \sigma_\text{p}(\mathcal{T}_\text{max} \mathcal{T}_\text{min}) = \sigma_\text{p}(- \Delta_D I_2)$ with multiplicity $2k$, i.e. if and only if $\lambda \in \sigma_\text{p}(-\Delta_D)$ with multiplicity $k$. The last observations together with~\eqref{spectrum_A_m_1} finish the proof of assertion~(iv).

\end{proof}

Let us end this note with a short discussion of the spectral properties of $A_m$ for some special domains $\Omega$ and some consequences of that.
In many situations it is known that the Dirichlet Laplacian has purely discrete spectrum. Then, by Theorem~\ref{theorem_spectrum_general}~(iv) also the spectrum of $A_m$ consists only of eigenvalues and, as a consequence of the spectral theorem, $m$ is an eigenvalue with infinite multiplicity. Moreover, in a similar way as sketched in the proof of Proposition~\ref{proposition_susy_spectrum}~(ii) one can construct eigenfunctions of $A_m$.
The spectrum of the Dirichlet Laplacian is purely discrete, e.g., when $\Omega$ is a bounded subset of $\mathbb{R}^3$, as then the space $H^1_0(\Omega; \mathbb{C})$ is compactly embedded in $L^2(\Omega; \mathbb{C})$ by the Rellich embedding theorem, and hence the Dirichlet Laplacian $-\Delta_D$ associated to the sesquilinear form $\mathfrak{a}_D$ in~\eqref{def_Dirichlet_form} has a compact resolvent. In this situation let us denote by $0 < \mu_1^D \leq \mu_2^D \leq \mu_3^D \leq \dots$ the discrete eigenvalues of $-\Delta_D$, where multiplicities are taken into account. Then one immediately has the following result.

\begin{cor} \label{corollary_spectrum_bounded_Omega}
  Let $\Omega \subset \mathbb{R}^3$ be such that $\sigma(-\Delta_D)$ is purely discrete. Then
  \begin{equation*} 
    \sigma(A_m) = \{ m \} \cup \Big\{ \pm \sqrt{m^2 + \mu_k^D}: k \in \mathbb{N} \Big\}
  \end{equation*}
  and $m$ is an eigenvalue with infinite multiplicity.
\end{cor}

If the Sobolev space $H^s(\Omega; \mathbb{C})$ is compactly embedded in $L^2(\Omega; \mathbb{C})$ for $s>0$, then the above result implies that $\dom A_m$ can not be contained in $H^s(\Omega; \mathbb{C}^4)$. This is, e.g., the case, when $\Omega$ is a bounded Lipschitz domain.

\begin{cor} \label{corollary_Sobolev_smoothness}
  Assume that $\Omega$ is a bounded subset of $\mathbb{R}^3$ with a Lipschitz-smooth boundary. Then $\dom A_m \not\subset H^s(\Omega; \mathbb{C}^4)$ for all $s>0$. 
\end{cor}

\appendix

\section{Spectrum of supersymmetric operators} \label{appendix}

Let $\big(\mathcal{H}_1, (\cdot, \cdot)_{\mathcal{H}_1} \big)$ and $\big( \mathcal{H}_2, , (\cdot, \cdot)_{\mathcal{H}_2} \big)$ be complex Hilbert spaces, let $S$ be a closed and densely defined operator from $\mathcal{H}_2$ to $\mathcal{H}_1$, and let $m \in \mathbb{R}$. In this appendix we consider the block operator $D_m$ acting in $\mathcal{H}_1 \oplus \mathcal{H}_2$ given by
\begin{equation} \label{def_D}
  D_m := \begin{pmatrix} m & S \\ S^* & -m \end{pmatrix}, \quad \dom D_m = \dom S^* \oplus \dom S \subset \mathcal{H}_1 \oplus \mathcal{H}_2.
\end{equation}
Note that $A_m$ defined in~\eqref{def_A_m} is exactly of the above form with $S = \mathcal{T}_\text{min}$. The goal in this appendix is to prove the self-adjointness of $D_m$ and to provide a useful formula for its spectrum. For this, we use that $D_m$ has a supersymmetric structure. Supersymmetric operators are well-studied, see, e.g., \cite{O02, P06, S91, T92}, and the results presented here seem to be well-known. For the sake of completeness, we present full proofs of the results in this appendix.

First, we show that $D_m$ is self-adjoint. In the proof, we use similar ideas as in~\cite[Proposition~1]{S95}.
\begin{prop} \label{proposition_susy_self_adjoint}
  The operator $D_m$ defined by~\eqref{def_D} is self-adjoint.
\end{prop}
\begin{proof}
  We use for $\Psi \in \mathcal{H}_1 \oplus \mathcal{H}_2$ the splitting
  $\Psi = (\psi_1, \psi_2)$ with $\psi_1 \in \mathcal{H}_1$ and $\psi_2 \in \mathcal{H}_2$. It suffices to consider $m = 0$,
  as $D_m-D_0$ is a bounded self-adjoint perturbation. 
  
  First we show that $D_0$ is symmetric. Indeed for $\Psi = (\psi_1, \psi_2) \in \dom D_0$ a simple calculation shows
  \begin{equation*}
    (D_0 \Psi, \Psi)_{\mathcal{H}_1\oplus \mathcal{H}_2} = (S \psi_2, \psi_1)_{\mathcal{H}_1} + (S^* \psi_1, \psi_2)_{\mathcal{H}_2}
    = 2 \, \text{Re}\, (S \psi_2, \psi_1)_{\mathcal{H}_1} \in \mathbb{R}.
  \end{equation*}
  Next, one has for $\Psi = (\psi_1, \psi_2) \in \dom D_0^*$ and $\Phi = (\phi_1, \phi_2) \in \dom D_0$
  \begin{equation} \label{equation_adjoint_cal_A}
    ( D_0^* \Psi, \Phi )_{\mathcal{H}_1 \oplus \mathcal{H}_2} = ( \Psi, D_0 \Phi )_{\mathcal{H}_1 \oplus \mathcal{H}_2}
    = (\psi_1, S \phi_2)_{\mathcal{H}_1} + (\psi_2, S^* \phi_1)_{\mathcal{H}_2}.
  \end{equation}
  Choosing $\phi_1 = 0$ we get from~\eqref{equation_adjoint_cal_A}
  \begin{equation*}
    \big( ( D_0^* \Psi)_2, \phi_2 \big)_{\mathcal{H}_2}
    = (\psi_1, S \phi_2)_{\mathcal{H}_1}
  \end{equation*}
  for all $\phi_2 \in \dom S$ and hence $\psi_1 \in \dom S^*$ and $S^* \psi_1 = ( D_0^* \Psi)_2$.
  In a similar way, choosing $\phi_2 = 0$ we obtain from~\eqref{equation_adjoint_cal_A} and the closeness of $S$ that
  $\psi_2 \in \dom S$ and $S \psi_2 = ( D_0^* \Psi)_1$.
  Therefore, we conclude $\Psi \in \dom D_0$ and $D_0^* \Psi = D_0 \Psi$,
  that means $D_0^* \subset D_0$. This finishes the proof of this proposition.
\end{proof}

In the following proposition we describe the spectrum of $D_m$ in terms of the spectra of $S^* S$ and $S S^*$. A variant of the result is contained in \cite{P06} for bounded $S$ and stated in \cite{O02} without proof and in \cite{S91} for a special choice of $S$. Here, we give the proof in the general situation, where similar ideas as in the above references are used. In the formulation of the result we use for a set $A \subset [0,\infty)$ and $b \in \mathbb{R}$ the notation $b \sqrt{A} := \{ b \sqrt{a}: a \in A \}$.

\begin{prop} \label{proposition_susy_spectrum}
  For the operator $D_m$ defined in~\eqref{def_D} the following is true:
  \begin{itemize}
    \item[$\textup{(i)}$] The spectrum of $D_m$ is
    \begin{equation} \label{susy_spectrum}
      \sigma(D_m) = \sign(m) \left(-\sqrt{\sigma(S^*S + m^2)} \cup \sqrt{\sigma(SS^* + m^2)} \right).
    \end{equation}
    \item[$\textup{(ii)}$] The point spectrum of $D_m$ is
    \begin{equation*}
      \sigma_\textup{p}(D_m) = \sign(m) \left(-\sqrt{\sigma_\textup{p}(S^*S + m^2)} \cup \sqrt{\sigma_\textup{p}(SS^* + m^2)} \right)
    \end{equation*}
    and the multiplicities of the eigenvalues coincide in the following sense:
    \begin{equation*}
      \begin{split}
        \text{if } \sqrt{\lambda + m^2}>0: \quad \dim \ker (S^* S - \lambda) &= \dim \ker \big(D_m + \sign(m) \sqrt{\lambda + m^2}\big); \\
        \text{if } \sqrt{\lambda + m^2}>0: \quad \dim \ker (S S^* - \lambda) &= \dim \ker \big(D_m - \sign(m) \sqrt{\lambda + m^2}\big); \\
        \dim \ker S^* S + \dim \ker SS^* &= \dim \ker D_0.
      \end{split}
    \end{equation*}
    \item[$\textup{(iii)}$] The essential spectrum of $D_m$ is
    \begin{equation*} 
      \sigma_\textup{ess}(D_m) = \sign(m) \left(-\sqrt{\sigma_\textup{ess}(S^*S + m^2)} \cup \sqrt{\sigma_\textup{ess}(SS^* + m^2)} \right).
    \end{equation*}
  \end{itemize}
\end{prop}
\begin{proof}
  In order to prove (i), we verify first that
  \begin{equation} \label{inclusion3}
    \sigma(D_m) \setminus \{ \pm m\}= -\sqrt{\sigma(S^*S + m^2)\setminus\{m^2\}} \cup \sqrt{\sigma(SS^* + m^2)\setminus\{m^2\}};
  \end{equation}
  note that $\sigma(S S^*) \setminus \{ 0 \} = \sigma(S^* S) \setminus \{ 0 \}$ holds, cf., e.g., \cite[Corollary~5.6]{T92}, and thus the set on the right hand side of the last equation is symmetric around the origin and independent of  $\sign(m)$.
  For the first inclusion 
  \begin{equation} \label{inclusion1}
    \sigma(D_m) \setminus \{ \pm m\} \subset -\sqrt{\sigma(S^*S + m^2)\setminus\{m^2\}} \cup \sqrt{\sigma(SS^* + m^2)\setminus\{m^2\}},
  \end{equation}
  we note that 
  \begin{equation} \label{D_square}
    D_m^2 = \begin{pmatrix} S S^* + m^2 & 0 \\ 0 & S^* S + m^2 \end{pmatrix}.
  \end{equation}
  Since $\sigma(S S^*) \setminus \{ 0 \} = \sigma(S^* S) \setminus \{ 0 \}$, cf. \cite[Corollary~5.6]{T92}, this implies~\eqref{inclusion1}.
  
  Next, we show that 
  \begin{equation} \label{inclusion2}
     \sqrt{\sigma(SS^* + m^2)\setminus\{m^2\}} \subset \sigma(D_m) \setminus \{ \pm m\}.
  \end{equation}
  For this purpose let $\lambda \in \sigma(SS^*) \setminus \{ 0 \}$. Then, there exists a sequence $(\psi_n) \subset \dom SS^*$ such that $\| \psi_n \|_{\mathcal{H}_1} = 1$ and $(SS^* - \lambda) \psi_n \rightarrow 0$, as $n \rightarrow \infty$. Define
  \begin{equation*}
    \Phi_n := \big( D_m + \sqrt{\lambda + m^2} \big) \begin{pmatrix} \psi_n \\ 0 \end{pmatrix} = \begin{pmatrix} (m + \sqrt{\lambda+m^2}) \psi_n \\ S^* \psi_n \end{pmatrix}.
  \end{equation*}
  Then $\Phi_n \in \dom D_m$, $\| \Phi_n \|_{\mathcal{H}_1 \oplus \mathcal{H}_2} \geq (m+\sqrt{\lambda+m^2}) \| \psi_n \|_{\mathcal{H}_1}=m + \sqrt{\lambda+m^2} > 0$ independently of $n$, as $\lambda > 0$ by assumption. Moreover, we have
  \begin{equation} \label{approximate_eigensequence}
    \begin{split}
      \big( D_m - \sqrt{\lambda + m^2} \big) \Phi_n &= \begin{pmatrix} m - \sqrt{\lambda+m^2} & S \\ S^* & -m-\sqrt{\lambda+m^2} \end{pmatrix} \begin{pmatrix} (m + \sqrt{\lambda+m^2}) \psi_n \\ S^* \psi_n \end{pmatrix} \\
      &= \begin{pmatrix} (SS^*-\lambda) \psi_n \\ 0 \end{pmatrix}.
    \end{split}
  \end{equation}
  Since the last expression converges to zero due to the properties of $\psi_n$, as $n \rightarrow \infty$, we conclude that $\sqrt{\lambda+m^2} \in \sigma(D_m)$ and therefore, \eqref{inclusion2} is true.

  In a similar way as above one verifies that 
  \begin{equation} \label{inclusion4}
     -\sqrt{\sigma(S^*S + m^2)\setminus\{m^2\}} \subset \sigma(D_m) \setminus \{ \pm m\}
  \end{equation}
  holds. Indeed, choose for $\lambda \in \sigma(S^*S) \setminus \{ 0 \}$ a sequence $(\phi_n) \subset \dom S^*S$ such that $\| \phi_n \|_{\mathcal{H}_2} = 1$ and $(S^*S - \lambda) \phi_n \rightarrow 0$, as $n \rightarrow \infty$. Then
  \begin{equation*} 
    \Psi_n := \big( D_m - \sqrt{\lambda + m^2} \big) \begin{pmatrix} 0 \\ \phi_n \end{pmatrix} = \begin{pmatrix} S \phi_n \\ (-m - \sqrt{\lambda+m^2}) \phi_n  \end{pmatrix}
  \end{equation*}
  satisfies $\| \Psi_n \|_{\mathcal{H}_1 \oplus \mathcal{H}_2} \geq m + \sqrt{\lambda+m^2} > 0$ independently of $n$ and
  \begin{equation*}
    \begin{split}
      \big( D_m + \sqrt{\lambda + m^2} \big) \Psi_n 
      &= \begin{pmatrix} 0 \\ (S^*S-\lambda) \phi_n 0 \end{pmatrix}.
    \end{split}
  \end{equation*}
  Since the last expression tends to zero due to the properties of $\phi_n$, as $n \rightarrow \infty$, we conclude that $-\sqrt{\lambda+m^2} \in \sigma(D_m)$ and therefore, \eqref{inclusion4} is true. Together with~\eqref{inclusion1} and  \eqref{inclusion2} this implies that~\eqref{inclusion3} holds.

  Next, we show for $m \neq 0$ that $m \in \sigma(D_m)$ if and only if $0 \in \sigma(S S^*)$. For this we note first that for $m \in \sigma(D_m)$ there exists a sequence $(\phi_n, \psi_n) \in \dom D_m$ with $\|\phi_n\|_{\mathcal{H}_1} + \|\psi_n\|_{\mathcal{H}_2}=1$ such that 
  \begin{equation*}
    (D_m-m) \begin{pmatrix} \phi_n \\ \psi_n \end{pmatrix} = \begin{pmatrix} 0 & S \\ S^* & -2m \end{pmatrix} \begin{pmatrix} \phi_n \\ \psi_n \end{pmatrix} = \begin{pmatrix}  S \psi_n \\ S^* \phi_n  -2m \psi_n \end{pmatrix} \rightarrow 0, \quad \text{ as } n \rightarrow \infty.
  \end{equation*}
  This is equivalent to $\|\phi_n\|_{\mathcal{H}_1} + \|\psi_n\|_{\mathcal{H}_2}=1$, $S \psi_n \rightarrow 0$ for $n \rightarrow \infty$, and 
  \begin{equation} \label{approximate_eigensequence2}
    \begin{split}
      0 &= \lim_{n \rightarrow \infty} \left\| S^* \phi_n  -2m \psi_n \right\|^2_{\mathcal{H}_2} \\
      &= \lim_{n \rightarrow \infty} \big( \| S^* \phi_n \|^2_{\mathcal{H}_2} - 4 m \, \text{Re}\, (S^* \phi_n ,\psi_n)_{\mathcal{H}_2} + 4 m^2 \| \psi_n \|^2_{\mathcal{H}_2} \big) \\
      &= \lim_{n \rightarrow \infty} \big( \| S^* \phi_n \|^2_{\mathcal{H}_2} - 4 m \, \text{Re}\, ( \phi_n , S \psi_n)_{\mathcal{H}_1} + 4 m^2 \| \psi_n \|^2_{\mathcal{H}_2} \big) \\
      &= \lim_{n \rightarrow \infty} \big( \| S^* \phi_n \|^2_{\mathcal{H}_2} + 4 m^2 \| \psi_n \|^2_{\mathcal{H}_2} \big),
    \end{split}
  \end{equation}
  where $S\psi_n \rightarrow 0$ and $\| \phi_n \|_{\mathcal{H}_1} \leq 1$ for $n \rightarrow \infty$ were used in the last step. Hence, $\| \phi_n\|_{\mathcal{H}_1} \rightarrow 1$ and $\| S^* \phi_n\|_{\mathcal{H}_2}^2 \rightarrow 0$ for $n \rightarrow \infty$, which shows  $0 \in \sigma(S S^*)$.
  
  Conversely, if $0 \in \sigma(SS^*)$, then there exists a sequence $(\phi_n) \subset \dom (SS^*)$ with $\| \phi_n \|_{\mathcal{H}_1}=1$ for all $n \in \mathbb{N}$ and $\lim_{n \rightarrow \infty} \| S^* \phi_n\|_{\mathcal{H}_2} = 0$. Then $(\phi_n, 0) \in \dom D_m$ and 
  \begin{equation*}
    (D_m-m) \begin{pmatrix} \phi_n \\ 0 \end{pmatrix} = \begin{pmatrix} 0 & S \\ S^* & -2m \end{pmatrix} \begin{pmatrix} \phi_n \\ 0 \end{pmatrix} = \begin{pmatrix}  0 \\ S^* \phi_n  \end{pmatrix} \rightarrow 0, \quad \text{ as } n \rightarrow \infty,
  \end{equation*}
  which shows $m \in \sigma(D_m)$. Hence, we have proved that $m \in \sigma(D_m)$ if and only if $0 \in \sigma(SS^*)$.
  
  For $m \neq 0$ the statement that $-m \in \sigma(D_m)$ if and only if $0 \in \sigma(S^* S)$ can be done in the same way as above. This finishes the proof of~\eqref{susy_spectrum} for $m \neq 0$.
  
  If $m=0$, then $0 \in \sigma(D_0)$ is equivalent to the existence of a sequence $(\phi_n, \psi_n) \in \dom D_0$ such that $\| \phi_n \|_{\mathcal{H}_1} + \| \psi_n \|_{\mathcal{H}_2} = 1$ and 
  \begin{equation*}
    D_0 \begin{pmatrix} \phi_n \\ \psi_n \end{pmatrix} = \begin{pmatrix} S \psi_n \\ S^* \phi_n \end{pmatrix} \rightarrow 0, \qquad \text{as } n \rightarrow \infty.
  \end{equation*}
  This is true if and only if (a) $\| \phi_n \|_{\mathcal{H}_1} \geq c$ for some $c>0$ and infinitely many $n$ and $\| S^* \phi_n \|_{\mathcal{H}_2} \rightarrow 0$, as $n \rightarrow \infty$, or (b) $\| \psi_n \|_{\mathcal{H}_2} \geq c$ for some $c>0$ and infinitely many $n$ and $\| S \psi_n \|_{\mathcal{H}_1} \rightarrow 0$, as $n \rightarrow \infty$. This is equivalent to (a) $0 \in \sigma(SS^*)$ or (b) $0 \in \sigma(S^* S)$. Together with~\eqref{inclusion3} this finishes the proof of (i) also in the case $m = 0$.

  To prove item~(ii),
  assume first that $\lambda > 0$ and let $\phi \in \ker (SS^*-\lambda)$. Then, similarly as in~\eqref{approximate_eigensequence} one sees that 
  \begin{equation*}
    \Phi := \big(D_m + \sqrt{\lambda + m^2}\big) \begin{pmatrix} \phi \\ 0 \end{pmatrix} = \begin{pmatrix} (m + \sqrt{\lambda+m^2}) \phi \\ S^* \phi  \end{pmatrix} \in \ker \big(D_m - \sqrt{\lambda + m^2}\big)
  \end{equation*}
  and due to the explicit form of $\Phi$ we see that 
  \begin{equation} \label{dim2}
    \dim \ker (SS^*-\lambda) \leq \dim \ker \big(D_m - \sqrt{\lambda + m^2}\big)
  \end{equation}
  holds. Similarly, if $\psi \in \ker (S^*S-\lambda)$, then 
  \begin{equation*} 
    \Psi := \big(D_m - \sqrt{\lambda + m^2}\big) \begin{pmatrix} 0 \\ \psi \end{pmatrix} = \begin{pmatrix} S \psi \\ -(m + \sqrt{\lambda+m^2}) \psi  \end{pmatrix} \in \ker \big(D_m + \sqrt{\lambda + m^2}\big)
  \end{equation*}
  and due to the explicit form of $\Psi$ we see that 
  \begin{equation} \label{dim3}
    \dim \ker (S^*S-\lambda) \leq \dim \ker \big(D_m + \sqrt{\lambda + m^2}\big).
  \end{equation}
  
  Next, in view of~\eqref{D_square} we have
  \begin{equation*}
    \left(D_m - \sqrt{\lambda+m^2}\right) \left(D_m + \sqrt{\lambda+m^2}\right)  = \begin{pmatrix}
      SS^* - \lambda & 0 \\ 0 & S^*S - \lambda
    \end{pmatrix}
  \end{equation*}
  and hence, 
  \begin{equation*}
    \begin{split}
      \dim \ker \left(D_m - \sqrt{\lambda+m^2}\right) &+ \dim \ker \left(D_m + \sqrt{\lambda+m^2}\right) \\
      & \leq  \dim \ker(SS^*-\lambda) + \dim \ker (S^*S-\lambda).
    \end{split}
  \end{equation*}
  Since one has for $\lambda > 0$ the relation $\dim \ker(SS^*-\lambda) = \dim \ker (S^*S-\lambda)$ by \cite[Corollary~5.6]{T92}, we conclude that this together with~\eqref{dim2} and~\eqref{dim3} implies
  \begin{equation*}
    \dim \ker \left(D_m - \sqrt{\lambda+m^2}\right) = \dim \ker \left(D_m + \sqrt{\lambda+m^2}\right)  =\dim \ker(SS^*-\lambda),
  \end{equation*}
  which is the statement of assertion~(ii) for $\lambda>0$.
  
  Next, we show for $m \neq 0$ that $(\phi, \psi) \in \ker(D_m-m)$ if and only if $\psi=0$ and $\phi \in \ker SS^*$. Indeed, $(\phi, \psi) \in \ker(D_m-m)$ is equivalent to
  \begin{equation*}
    (D_m-m) \begin{pmatrix} \phi \\ \psi \end{pmatrix} = \begin{pmatrix}  S \psi \\ S^* \phi  -2m \psi \end{pmatrix} = \begin{pmatrix} 0 \\ 0 \end{pmatrix},
  \end{equation*}
  i.e. $S \psi = 0$ and $S^* \phi  -2m \psi=0$. As in~\eqref{approximate_eigensequence2} one finds that this is equivalent to
  $\psi =0$ and $S^* \phi = 0$. Thus, $\dim \ker(D_m-m) = \dim \ker SS^*$.
  
  For $m \neq 0$ the statement that $\dim \ker(D_m+m) = \dim \ker S^*S$ can be shown in the same way as $\dim \ker(D_m-m) = \dim \ker SS^*$.
  
  Eventually, if $m=0$, then $(\phi, \psi) \in \ker D_0$ if and only if
  \begin{equation*}
    D_0 \begin{pmatrix} \phi \\ \psi \end{pmatrix} = \begin{pmatrix} S \psi \\ S^* \phi \end{pmatrix} = \begin{pmatrix} 0 \\ 0 \end{pmatrix}.
  \end{equation*}
  This implies immediately $\dim \ker S^*S + \dim \ker SS^* = \dim \ker D_0$. Hence, statement (ii) has been shown in all cases.
  
  Finally, assertion~(iii) is an immediate consequence of the results from~(i) and~(ii).
\end{proof}

\vskip 0.8cm
\noindent {\bf Acknowledgments.}  
The author thanks Jussi Behrndt, Andrii Khrabustovskyi, and Peter Schlosser for helpful discussions. Moreover, helpful remarks from the anonymous reviewers are gratefully acknowledged.


\vskip 0.8cm
\noindent {\bf Data availability statement.}  
Data sharing not applicable to this article as no datasets were generated or analysed during the current study.



\end{document}